\documentclass{amsart}
\usepackage[dvips]{color}
\usepackage{graphicx}

\usepackage{latexsym}
\usepackage{amssymb}
\usepackage{amsmath}
\usepackage{amsthm}
\usepackage{amscd}
\usepackage{graphicx}
\usepackage{color}
\usepackage{latexsym}

\usepackage{enumerate}
\usepackage{verbatim}
\usepackage{xy}
\xyoption{all}
\theoremstyle{plain}
\newtheorem{thm}{Theorem}
\newtheorem{lem}{Lemma}
\newtheorem{Bob}{Definition}
\newtheorem{prop}{Proposition}

\theoremstyle{remark}
\newtheorem{rem}{Remark}

\newtheorem{ex}{Example}
\newtheorem{ques}{Question}
\newcommand\R{\mathbb R}
\newcommand\mbc[1]{}
\newcommand{\constant}{\tfrac 1 {12}}
\newcommand{\constantinv}{12}
\newcommand{\fracpartt}[1]{\langle\!\langle#1\rangle\!\rangle}
\newcommand{\ui}{\ensuremath{\mathcal{I}}}
\newcommand{\Leb}{\ensuremath{\lambda}}

\newcommand{\Limsup}{\limsup_{n\to\infty}\limits}

\newtheorem{cor}{Corollary}
\begin{document}
\title{Borel-Cantelli Sequences\\ \tiny{\today}}\
\author[M. Boshernitzan]{Michael Boshernitzan}
\address{Department of Mathematics, Rice University, Houston, TX~77005, USA}
\email{michael@math.rice.edu}
\author[J.\ Chaika]{Jon Chaika}

\address{Department of Mathematics, Rice University, Houston, TX~77005, USA}

\email{Jonathan.M.Chaika@rice.edu}

\maketitle

\begin{abstract}
A sequence $\{ x_{n}\}_1^\infty$ in the unit interval  $[0,1)=\mathbb{R} / \mathbb{Z}$ is called \mbox{{\em Borel-Cantelli}} (BC) if for all
non-increasing sequences of positive real numbers $\{a_n\}$ with
$\underset{i=1}{\overset{\infty}{\sum}}a_i=\infty$\,
\mbox{the set}
\[
\big\{x\in[0,1)\ \big|\ |x-x_n|<a_n \text{ for infinitely many }n\geq1\big\}
\]
has full Lebesgue measure.
(To put it informally, BC sequences are sequences for which a natural converse to the Borel-Cantelli Theorem holds).

The notion of BC sequences is motivated by the Monotone Shrinking Target Property for dynamical systems,
but our approach is from a geometric rather than dynamical perspective. A sufficient condition, 
a necessary condition and a necessary and sufficient condition for a sequence to be BC are established. 
A number of examples of BC and not BC sequences are presented.

The property of a sequence to be BC is a delicate diophantine property.
For example, the orbits of a pseudo-Anosoff IET (interval exchange transformation) are BC while the orbits 
of a ``generic'' IET are not.

The notion of BC sequences is extended to sequences in Ahlfors regular spaces (rather than in $[0,1)$).
\end{abstract}
\section{Set up}
Denote by\, $\ui=[0,1)=\mathbb{R} / \mathbb{Z}$  the unit interval and by $\Leb$ the Lebesgue measure on it.
For $r>0$ and $a\in\ui$, denote by $B(a,r)$ the $r$-ball around  $a$
(taken mod $1$, so that $\Leb(B(c,r))=\min(2r,1)$). For $c \in \mathbb{R}$, let $\fracpartt{c}=c-\big[c\big]\in\ui$
denote the fractional part of $c$ (or $c$ mod $1$).
\vspace{1mm}


By a  {\em standard sequence} we mean a non-increasing sequence\, ${\bf a}=\{a_n\}_1^\infty$
of positive real numbers \mbox{$a_1\geq a_2\geq\cdots\,$} converging to $0$ such that
$\underset{n=1}{\overset{\infty}{\sum}}a_n=\infty$.
\begin{Bob}\label{def:BCui}
A sequence\,  ${\bf x}=\{x_n\}_1^\infty$ in\, $\ui=[0,1)$  is called
Borel-Cantelli (BC)  if\, $\Leb\big(\limsup_{n\to\infty}\limits B(x_n,a_n)\big)=1$
for every standard sequence ${{\bf a}=\{a_n\}_1^\infty}$.
\end{Bob}

Recall that\, $\Limsup B(x_n,a_n)={\underset{k=1}{\overset{\infty}{\cap}}
\underset{n=k}{\overset{\infty}{\cup}} B(x_n, a_n)}$
denotes the set of all points lying in infinitely many $B(x_n,a_n)$.

Observe that if\, $\underset{i=1}{\overset{\infty}{\sum}}a_i<\infty $ then
${\Leb(\Limsup B(x_n, a_n))=0}$
by the Borel-Cantelli Theorem.
Also note that if we fix a dense sequence $\bf{x}$ and let
\[
a_i=
\begin{cases}
\frac15,& \text{if }\ \frac25<x_{i}<\frac35\\
0,& \text{otherwise,}
\end{cases}
\]
then $\underset{i=1}{\overset{\infty}{\sum}}a_i=\infty$\,  but
$\Leb(\Limsup B(x_n, a_n))=\Leb((\frac 1 5,\frac 4 5))<1.$
 This example shows that one must restrict the choice of radii in some way beyond the obvious condition given by the Borel-Cantelli Theorem. Restricting to non-increasing targets is reasonable because it is a mild condition, many sequences are Borel-Cantelli, and it is natural in the context of dynamical systems as seen by the Monotone Shrinking Target Property (MSTP) (see the survey paper \cite{ja}).


\begin{Bob} A $\Leb$ measure preserving map $T \colon \ui \to \ui$ satisfies the
\emph{Monotone Shrinking Target Property} (MSTP) if for  any standard sequence $\bf a$
and any $y \in \ui$
\[
\Leb\big(\Limsup T^{-n}(B(y, a_n))\big)=0.
\]
\end{Bob}

We can consider a dual property. We say a map $T\colon \ui \to \ui$ is Absolutely
Borel-Cantelli (ABC) if the forward orbit $\{T^nx\}_{n\geq0}$ of every point
$x\in \ui$ is BC.
(That is, for any $x\in X$ and any standard sequence
${\bf{a}}=\{a_k\}_1^\infty$ the relation $\Leb(\Limsup B(T^nx, a_n))=1$ holds).

The emphasis in our paper is on abstract sequences, not necessarily originating from
dynamical systems.
 We focus on the Borel-Cantelli property (for sequences) as a version of the ABC
 property (for maps).
Note that we don't have a natural candidate for the notion of MSTP for abstract sequences.

 Approximation of points in a space by sets have also been considered in the context of regular systems \cite{reg sys} and ubiquitous systems \cite{BDV}. Some of our results (sufficient conditions Theorems \ref{suff} and \ref{suff gen}) have been proven more generally in these contexts. In particular, ubiquitous systems considers approximation by sets (instead of just points) and allows for more general targets. This generality leads to much more involved definitions (than BC). At least one natural example of approximation by sets can also be handled by BC sequences (Example \ref{farey} can be thought of as describing approximation of irrationals by rationals based on denominator). In this paper, we provide necessary and sufficient conditions for the conclusions of certain of the results in the context we consider (see Remark \ref{ubiq}). 

The BC (Borel Cantelli) property is quite delicate as the following examples
suggest.

Let $\alpha\in\R$.
If $x_n=\fracpartt{n\alpha}$ then $\bf x$ is BC  if and only if $\alpha$ is badly
approximable irrational, i.\,e. if the terms in its continued fraction expansion are bounded
(see \cite{kurz} and also Example \ref{classic}). If $x_n=\fracpartt{\alpha\log(n)}$ then $\bf x$ is
BC for any $\alpha \neq 0$ (see Example~\ref{log}).
If $x_n=\fracpartt{\alpha\sqrt{\log(n)}}$ then $\bf x$ is never BC
(Corollary~\ref{small sep}).

In this paper we also show that a number of natural sequences are BC.
These include sequences given by some (but not other) independent identically distributed
random variables (see Examples \ref{RV1}, \ref{RV2} and \ref{RV gen}).
The Farey sequence of rationals (taken in the natural order) is also BC.
(This observation recovers a classic theorem of Khinchin on approximation of irrationals by rationals,
see Example~\ref{khinch}).
Additionally, $x_n=\fracpartt{\sqrt{n}}$  is BC by the results in \cite{EM}
concerning the distribution of gaps of this sequence and $x_n= \fracpartt{n^2\alpha}$ is BC for almost every $\alpha$ due to weaker results on gaps in \cite{RS} (see Remark \ref{gaps}). On the other hand, the same sequence $x_n=\fracpartt{n^2\alpha}$ fails to be BC for a residual set
of $\alpha$, in particular for all $\alpha$  satisfying \
$\inf_{n\geq1}\limits\fracpartt{\alpha n^4}=0$.

 We conjecture, and some computer computations suggest, that a large class of sequences like $\{\fracpartt{\sqrt[3]{n} }\}$, $\{\fracpartt{n \log(n)}\}$ and $\{\fracpartt{(\log(n))^2}\}$ are BC; however, we are lacking the rigorous methods to validate this conjecture.

In the context of dynamical systems, systems satisfying a mild quantitative rigidity
condition have almost every orbit not BC (see Corollary \ref{quant rig}).
These conditions show that for almost every IET $T$ almost all orbits are not BC
(see \mbox{\cite[Theorem 7]{kurz iet}}).

On the other hand, linearly recurrent systems are ABC \mbox{(Example~\ref{lin rec bc})}.
The result implies that some exceptional IETs (like the pseudo-Anosoff, or self-similar
ones) are ABC  (see Example~\ref{lin rec IET}).  In particular, all minimal IETs over
quadratic number fields are ABC  (because these reduce to a pseudo-Anosoff IET on
a subinterval by \cite[Proposition 1]{BoshCar}).

The main results of the paper are:
\begin{enumerate}
\item A frequently checkable sufficient condition for a sequence to be BC (Theorems \ref{suff} and \ref{suff gen}).
\item A frequently checkable necessary condition (Theorems \ref{nec} and \ref{nec gen}), which is phrased
      as sufficient condition for a sequence not to be BC.
\item A necessary and sufficient condition (Theorems \ref{nec suff} and \ref{nec suff gen}).
\end{enumerate}

 The first two conditions and their corollaries help determine whether or not many sequences are Borel-Cantelli. The last condition provides some properties of BC sequences and identifies the properties that govern whether or not a sequence is BC (see Remark \ref{ybc}). These results are proven for $\ui$ and then generalized to Ahlfors regular spaces (Section 3). The methods in this paper are robust and can be applied to other related situations (see Remark \ref{other} and Section 4).

The plan for this paper is to address first the results for sequences $[0,1)$
(which are most developed in the dynamical side of the literature) in the second section.
The Borel-Cantelli status of many natural sequences is addressed in this section.
In the  third section we generalize these results (from the unit interval) to Ahlfors regular spaces.
We generalize these results to some weaker properties in the fourth section.
Then we present some classification results in the fifth section.
The main tools of this paper are density point arguments and covering arguments.
Throughout this paper constants are found, though they are not optimal.

\section{[0,1) and Lebesgue Measure}
The following theorem provides a checkable sufficient condition for a sequence to be BC (Borel-Cantelli).
The condition will also be used in the proof of Theorem \ref{nec suff}, the necessary and sufficient
condition for a sequence to be BC.
\begin{thm} \label {suff}
Let ${\bf{x}}=\{x_n\}_1^\infty$ be a sequence in the unit interval $\ui$ and assume that
there exists $d>0$ such that \
$\underset{N \to \infty}{\liminf}\,\Leb(\underset{i=1}{\overset{N}{\cup}} B(x_i, \frac 1 N) \cap J) \geq d \Leb(J)$
for all intervals $J$.  Then $\bf x$ is BC.
\end{thm}
The proof will follow Corollary \ref{link2}.
\begin{rem} This result is analogous to results for regular systems by 
V.\,Beres\-ne\-vich~\cite{ber russian}. 
We include the proofs for completeness.
\end{rem}
\begin{ex} \label{RV1} If $\{R_{n}\}$ is a sequence of independent random variables,
all distributed according to a probability measure $\mu$ that has Radon-Nikodym
derivative bounded away from $0$, then for $\mu^{\mathbb{N}}$ almost every $\zeta$
the sequence $\{R_n(\zeta)\}$ is BC.
(See also Example \ref{RV2} for a more precise result).

It is classical and not hard to show that for any particular sequence of
positive reals ${\bf a}=\{a_n\}$ (not necessarily monotone) with
$\sum_1^\infty a_{n}=\infty$ almost every sequence $\{R_n(\zeta)\}$ satisfies
$\Leb\big(\Limsup B(R_n(\zeta),a_n)\big)=1$.
Theorem \ref{suff} claims that a full measure set works {\em simultaneously}
for all standard sequences.
\end{ex}
\begin{cor} \label{dense}
If there exists $D>0$ such that the sets\, ${\bf X_n}=\{x_k\mid 1\leq k\leq n\}$\,
are $\frac D n$-dense in $\ui$ for all large enough $n$ then the sequence $\bf x$ is BC.
\end{cor}
\begin{ex}\label{log} It follows that if $x_i=\fracpartt{\log_c(i)}$ then
   the sequence  $\bf x$ is BC because the sets  $\bf X_n$  are roughly $\frac 1 {nc \ln c}$ dense.
\end{ex}
\begin{ex}\label{lin rec IET}
It follows from Corollary \ref{dense} that linear  recurrent IETs are ABC, i.\,e. that
every forward orbit is BC. An interval exchange transformation is called linearly
recurrent if its symbolic coding is a linearly recurrent subshift (see \cite{lin ref}
for introduction and basic properties of linearly recurrent subshifts).
\end{ex}
\begin{cor} \label{farey} If $\bf x$ is uniformly distributed and there exists $c$
such that ${|x_n-x_m|> \frac c {\max \{n,m\}}}$ then $\bf x$ is BC.
\end{cor}
\begin{ex} \label{khinch} Define the Farey sequence by the rational numbers in $\ui$
arranged in the following order
\[
 0,1,\tfrac12,\tfrac13,\tfrac23,\tfrac14,\tfrac34,\tfrac15,\tfrac25,\tfrac35,\tfrac45,
 \tfrac16,\tfrac56,\cdots
\]
Corollary \ref{farey} implies that the Farey sequence is BC. (Note that $\frac p q$ is the O$(q^2)$-th term in this sequence).
The fact that  Farey sequence is BC easily implies Khinchin's classic theorem which states that\\
\[
\Leb\Big(\Big\{\alpha\in\ui: \Big|\alpha-\frac p q\Big|<\frac{a_q}{q} \text{ for infinitely many } q\Big\}\Big)=1
\]
\\
for any standard sequence ${\bf{a}}=\{a_{k}\}$
(see e.\,g.\ \cite[Theorem~32]{best}  for a slightly weaker result).
The reduction is based on the observation that
$\underset{i=1}{\overset{\infty}{\sum}}ia_i=\infty$ if and only if
$\underset{i=1}{\overset{\infty}{\sum}}a_{\lfloor \sqrt{i} \rfloor}=\infty$.
Indeed,
$
\underset{i=1}{\overset{\infty}{\sum}} M^{2(i-1)}a_{M^i}\leq \underset{i=1}{\overset{\infty}{\sum}}ia_i\leq \underset{i=0}{\overset{\infty}{\sum}} M^{2(i+1)}a_{M^{i}},
$
and, on the other hand,
$\underset{i=1}{\overset{\infty}{\sum}} M^{2(i-1)}a_{M^i}\leq \underset{i=1}{\overset{\infty}{\sum}}a_{\lfloor \sqrt{i} \rfloor}\leq \underset{i=0}{\overset{\infty}{\sum}} M^{2(i+1)}a_{M^{i}}.
$
\end{ex}


The following is a checkable necessary condition for a sequence to be BC
(sufficient condition for a sequence not to be BC).
It is a \emph{partial} converse to Theorem~\ref{suff}.\\
\begin{thm}\label{nec}
If there exists an interval $J$ such that for every $\epsilon>0$ there exists arbitrarily large
$N_{\epsilon}$ with \ $\Leb(\underset{i=1}{\overset{N_{\epsilon}}{\cup}}
B(x_i,\frac 1 {N_{\epsilon}}) \cap J)<\epsilon\Leb (J)$ \
then $\bf x$ is not BC.
\label{nec}
\end{thm}
\begin{proof}
We can choose $N_{2^{-i}}< N_{2^{-(i+1)}}$. Let 
\[
a_i= \tfrac{1}{N_{2^{-i}}} \ \text{ for all } \ N_{2^{-(i+1)}}\leq i<N_{2^{-i}}.
\]
$\bf a$ is a standard sequence (because it is non-increasing and there are $N_{2^{-i}}-N_{2^{-(i+1)}}$ terms of size $\frac 1 {N_{2^{-i}}}$.) 
By hypothesis there exists $J$ such that \mbox{$\underset{n=1}{\overset{\infty}{\sum}} \Leb(B(x_n,a_n)\cap J)<\infty$}. 
Therefore, by the Borel-Cantelli Theorem, $\Leb(\Limsup B(x_n,a_n)\cap J)=0$.
\end{proof}
From this we get the following general result for dynamical systems with a mild quantitative rigidity assumption.
\begin{cor} \label{quant rig} If $T\colon [0,1) \to [0,1)$ is $\Leb$ measure preserving and assume that
${\underset{n \to \infty}{\liminf} \, n \int_0 ^1 |T^nx-x| dx=0}$ then almost every forward orbit $\{T^n(x)\}$ is not BC
(i.\,e.,  the sequence $\{T^n(x)\}$ is not BC for almost all $x\in\ui$).
\end{cor}
\begin{proof} Choose $n_i$ such that $\int_0 ^1 |T^{n_i}x-x| dx<\frac {20^{-i}} {n_i}$. Observe that for each $j$
$\Leb(\{x:|T^{n_i+j}(x)-T^j(x)|>\frac1{2^in_i}\})<10^{-i}$. Thus $\Leb( \{x: |T^{kn_i+j}(x)-T^j(x)|>\frac k {2^in_i} \text{ for each } 1\leq k\leq i\})< i10^{-i}$. Therefeore, the Borel-Cantelli Theorem implies that for almost every $x$ the sequence $\{T^n(x)\}$ satisfies the condition of 
Theorem~\ref{nec}. 
\end{proof}
It follows from \cite[Part I, Theorem 1.4]{metric} that for almost every interval exchange transformation and almost every $x$ the sequence $\{T^n(x)\}$ is not BC. (One gets that the conditions of the corollary holds on sets of increasing positive measure that cover Lebesgue almost all of the interval.) One can tweak the argument to get that in this case \emph{every} orbit is not BC. See \cite[Theorem 7]{kurz iet} for the details.

\begin{ex}\label{classic} An immediate consequence of Corollaries  \ref{dense} and \ref{quant rig} is that $\{\fracpartt{n\alpha}\}$
is BC if and only if
the real $\alpha$ is a badly approximable irrational (that is, the terms of its continued fraction expansion are uniformly bounded).
The above claim is a restatement of a result originally proven by J.\ Kurzwiel in \cite{kurz}.
\end{ex}
\begin{cor}\label{small sep} If  ${\bf x}=\{x_n\}$ is a sequence in $\ui$ such that\,
$\underset{n \to \infty}{\limsup} \ nd(x_n,x_{n+1})=0$ then $\bf{x}$ is not BC.
\end{cor}
The corollary easily follows from Theorem \ref{nec}.
\begin{ex} It also follows that if $x_n=\fracpartt{\ln(\ln (3+n))}$ (or even $\fracpartt{(\ln(2+n))^{\frac {99}{100}}}$) then $\bf{x}$ is not BC.
\end{ex}
To state the necessary and sufficient condition for a sequence to be Borel-Cantelli a definition is required.
\begin{Bob} Let $A=\{N_n\}$ be an infinite increasing sequence of natural numbers. Given ${\bf x}=\{x_{n}\}$,
 define $f_A(z):= \underset{ r \to 0^+}{\liminf} \ \underset{N \in A }{\limsup}\  \frac{\Leb(\underset{i=1}{\overset{N}{\cup}}B(x_i, \frac 1 N) \cap B(z,r))}{\Leb(B(z,r))}$.
\end{Bob}
\begin{lem} $f_A$ is measurable.
\end{lem}
\begin{proof} Let $f_{A,r}(z) = 
\underset{N \in A}{\limsup}\frac{\Leb(\underset{i=1}{\overset{N}{\cup}}B(x_i, \frac 1 N) \cap B(z,r))}{\Leb(B(z,r))}$. 
$f_{A,r}$ is continuous. Also, $f_{A,r+\epsilon}(z)+2 \epsilon \geq f_{A,r}(z) \geq f_{A,r+\epsilon}(z) \frac {r}{r+\epsilon}$. 
Thus, $f_A(z)=\!\underset{r \in \mathbb{Q}, r \to 0^+}{\liminf}f_{A,r}(z)$ \ and is therefore measurable.
\end{proof}
\begin{thm} \label{nec suff}
A sequence ${\bf x}=\{x_{n}\}$ is not BC if and only if\, $\Leb(f_A^{-1}(0))\!>\!0$\, for some 
sequence $A$.
\end{thm}

We defer the proof of this theorem to the end of the section and first state some consequences.
\begin{rem} \label{ybc} The theorem shows that the BC property  can be detected by sequences of the form $a_i=\frac 1 {N_j}$ for $N_{j-1}<i\leq N_j$.
For the purposes of testing of the BC property one need not bother with the many standard sequences such that
$\underset{n\to \infty}{\limsup} \, na_n=0$ (such as $a_n=\frac 1 {n \ln (n)}$).
\end{rem}
\begin{rem} If one were to define $\tilde{f}(z)=\underset{ r \to 0^+}{\limsup} \ \underset{N \to \infty}{\liminf}\ \frac{\Leb(\underset{i=1}{\overset{N}{\cup}}B(x_i, \frac 1 N) \cap B(z,r))}{\Leb(B(z,r))}$, then there are BC sequences such that $\tilde{f}(z)=0$ for almost every $z$.
\end{rem}
\begin{rem} If we choose $A=\{1,2,3,...\}$ then there exist non-BC sequences such that $f_A=1$ almost everywhere.
\end{rem}

\begin{ex} It follows that a non uniquely ergodic IET has orbits that are not BC. Additionally, if $T: [0,1) \to [0,1)$ is a continuous, $\Leb$ measure preserving transformation that is not $\Leb$ ergodic then $\Leb$ almost every orbit is not BC.
\end{ex}
\begin{ex}\label{RV2} It follows that if $\{R_n\}$ is a sequence of independent random variables each distributed
according to a measure $\mu$, then $\{R_1(\zeta), R_2(\zeta),...\}$ is BC for $\mu^{\mathbb{N}}$ almost every $\zeta$ if and only if  $\Leb \ll \mu$.
\end{ex}

\begin{rem}\label{gaps} Fix a uniformly distributed mod 1 sequence $\bf x$. The first $n$ points define a partition of the $[0,1)$ into segments of length $\delta_1^{(n)},...,\delta_{n+1}^{(n)}$. It follows from Theorem~\ref{nec suff} that if for any $\epsilon>0$ there exists a constant $s_{\epsilon}>0$ such that for large $n$, all but $\epsilon n$ of the $\delta_i ^{(n)}$ are bigger than $\frac{s_{\epsilon}}{n}$ then $\bf x$ is Borel-Cantelli.

The above (sufficient) criterion applies to conclude that the square root sequence $\{\fracpartt{\sqrt{k}}\}_{k\geq1}$  is BC.  
The validation of the criterion is based on the gap distribution results for this sequence by N.~Elkies and C.\ McMullen
\cite[Theorem 1.1]{EM}. Likewise, the sequence $\{\fracpartt{k^r \alpha}\}_{k \geq 1}$ is BC for any integer $r \geq 2$ 
and almost every $\alpha$ by Z.~Rudnick and P.\ Sarnak \cite[Theorem 1]{RS}.
\end{rem}

The following is a stronger sufficient condition (than the one in Remark \ref{gaps})  for a sequence $\bf x$ to be  BC  which 
is easier to apply in some situations.
\begin{rem} \label{pairs} Let $\bf x$ be a sequence uniformly distributed in $\ui=[0,1)$. Denote 
\[
X_n(u)=\{(p,q)\mid 1\leq p<q\leq n,\, |x_p-x_q|<u\}, \quad \text{for\, $n\geq1$,\,$u>0$}.
\]
Assume that for any $\epsilon>0$ there exists a constant $s_{\epsilon}>0$ 
such that, for all large $n$, the cardinality of the set  $X_n(\frac{s_{\epsilon}}{n})$
does not exceed  $\epsilon n$.  Then  $\bf x$ is Borel-Cantelli.
\end{rem}

We now begin the proofs of Theorems \ref{suff} and \ref{nec suff} with the key lemma of this paper. It is distilled from the proof of \cite[Lemma 4]{kurz}.
\begin{lem} \label{key} Let $M \in \mathbb{N}, c>0, e>0$ be constants,
let $\bf{x}$ be a sequence in $\ui$ and $\bf a$ be a standard sequence. If for all $r\in \mathbb{N}$ at least $cM^r$ of the points in the set
$\{x_{M^{r-1}}, x_{M^{r-1}+1},..., x_{M^r}\}$ are $\frac e {M^r}$ separated from each other, then there exists \mbox{$\delta>0$}
depending only on $ c$ and $e$ such that $\Leb (\Limsup B(x_n, a_n) )> \delta $. In particular $\delta$ is independent of
 $\bf a$ (so long as $\bf a$ is standard).
\end{lem}
\begin{rem} \label{other} If one imposes stricter conditions on $\bf a$ then one can prove versions of this lemma with weaker hypotheses. For instance, if one wishes that $\bf a$ is standard and $ia_i$ is monotone then one only needs to assume that $cM^r$ of the points $\{x_{M^{r-1}}, x_{M^{r-1}+1},..., x_{M^r}\}$ are $\frac e {M^r}$ separated from each other for a positive (lower) density set of $r$. Call such a sequence a \emph{Khinchin sequence}. This approach is carried out to prove \cite[Theorem 8]{kurz iet}. One can prove versions of Theorems \ref{suff}, \ref{nec} and \ref{nec suff} in this context. In particular, for Khinchin sequences one can obtain the analogue to Theorem 3 by letting $f_A(z)=\underset{r \to 0^+}{\liminf} \ \underset{N \in A}{\limsup} \ \frac 1 N \underset{k=1}{\overset{N}{\sum}} \frac{\Leb(\underset{i=1}{\overset{2^k}{\cup}} B(x_i,\frac 1 {2^k}) \cap B(z,r))}{2r}$. In the more involved direction the proof is similar (using the fact that for Khinchin
  sequences one may apply Lemma \ref{simple} twice), and in the other direction let $a_i=\frac 1 {i \log(N_j)}$ for $N_{j-1} \leq i<N_j$. This can also be carried out in the Ahlfors regular setting.
\end{rem}

The following well known and simple fact is used in the proof of Lemma \ref{key}. 
\begin{lem} \label{simple} Let $M\geq2$ be an integer and $\bf a$ be a non-increasing sequence. Then $\underset{i=0}{\overset{\infty}{\sum}} a_i$ diverges if and only if\, $\underset{i=0}{\overset{\infty}{\sum}} M^{i-1} a_{M^i}$ diverges.
\end{lem}

\begin{proof}[Proof of Lemma {\rm\ref{key}}]  WLOG assume that $c> \frac 2 {M}$. We may do this by replacing $M$ with a power of $M$ (the new $c$ is $\frac {M^k-1} {M^k} c$ which can be greater than $\frac 2 {M^k}$ for some $k$). To ease in computation we replace $a_1,a_2,...$ by $b_1,b_2,...$ where $b_i = \min \{a_{M^j}, \frac e {2M^j} \}$ for $ M^{j-1} \leq i <M^j$. It suffices to show that for any $k_0$, $\Leb(\underset{i=k_0}{\overset{\infty}{\cup}}B(x_i,b_i))>\delta:=\frac{e c}2$.

Observe that a $\frac{e}{2M^j}$ neighborhood of $B(z,r)$
 can contain at most $\lceil 2r \frac{M^j}{e} \rceil+1$
 points that are $\frac e {M^j} $ separated.
 If $\Leb(\underset{i=k_0}{\overset{M^{j-1}}{\cup}} B(x_i,b_i))<\delta$ then at most $$2M^{j-1}+ \tfrac{M^j}{e}\delta=2M^{j-1}+\tfrac c 2 M^{j}$$ of the separated points from $\{x_{M^{j-1}},x_{M^{j-1}+1},...,x_{M^j}\}$ have $B(x_t,b_t)\cap \underset{i=k_0}{\overset{M^{j-1}}{\cup}} B(x_i,b_i)) \neq \emptyset$.
This leaves at least $(\frac c 2-\frac1 M)M^j$ separated points. 
This is positive because $c>\frac 2 M$. This means that 
\[
\Leb\big( (\underset{i=M^{j-1}}{\overset{M^{j}}{\cup}} B(x_i,b_i)\, \backslash
\underset{i=k_0}{\overset{M^{j-1}}{\cup}} B(x_i,b_i)\big) 
\geq \big(\tfrac c 2-\tfrac1 M\big)\!\!\underset{i=M^{j-1}}{\overset{M^j}{\sum}} b_i.
\]

Since $\underset{i=1}{\overset{\infty}{\sum}}b_i=\infty$,  the inequality \
$
{\underset{j=g}{\overset{\infty}{\sum}} \Big((\tfrac c2-\tfrac 1 M)\underset{i=M^{j-1}}{\overset{M^j}{\sum}} b_i \Big)>\delta}
$ \
holds for any $g$.  The proof is completed.
\end{proof}

The following local version is an immediate corollary.
\begin{cor} \label{loc key} Let $M \in \mathbb{N}, c>0, e>0$ be constants,  $\bf{x}$ be a sequence
in $\ui$ and $\bf a$ be a standard sequence. If there exists an interval $J$ such that for all
$r\in \mathbb{N}$ at least $c\Leb(J)M^r$ of the points in the set
$\{x_{M^{r-1}}, x_{M^{r-1}+1},..., x_{M^r}\}$ are $\frac e {M^r}$ separated from each other, and
lie in $J$ then there exists $\delta>0$ depending only on $ c$ and $e$ such that
${\Leb(\Limsup B(x_n, a_n) \cap J)> \delta\Leb(J)}$.
\end{cor}
\begin{lem}\label{link} Let $J\subset \ui$ be an interval and assume that for some  $d>0$ the inequality
$\underset{N \to \infty}{\liminf} \ \Leb(\underset{i=1}{\overset{N}{\cup}} B(x_i, \frac 1 N) \cap J) \geq d\Leb(J)$
holds. Then 
\[
\Leb(\Limsup B(x_n,a_n) \cap J)>\constant\, d\,\Leb(J).
\]
\end{lem}
\begin{proof} For simplicity for each $N$ we ignore the effect of the at most 2 different $B(x_i,\frac 1 N)$ where
 $B(x_i,\frac 1 N)\cap J\neq \emptyset$ but
 $B(x_i, \frac 1 N) \not
 \subset J$.
 Notice that if $\{y_1,y_2,...,y_k\}$ does not contain 2 points that are $\frac 1 N$ separated
then $\Leb(\underset{i=1}{\overset{k}{\cup}} B(y_i,\frac 1
N)) \leq \frac 3 N$. This is because $y_1,...,y_k$ are contained in
an interval of length at most $\frac 1 N$. 

It follows if $\frac 1 3
\Leb(\underset{i=1}{\overset{N}{\cup}} B(x_i,\frac 1 N)) \geq g$
then $\{x_1,x_2,...,x_N\}$ contains at least $g$ points that are
$\frac 1 N$ separated.

By our assumption for any $\epsilon>0$ and all large enough $N$ the
set $\{x_1,x_2,...,x_N\} \cap J$ contains at least $N (d
\Leb(J)-\epsilon)$ points. Let $M \geq \frac 12 {d \Leb(J)}$. For large
enough $r$ the set $\{x_1,...,x_{M^{r}}\}$ contains at least $M^r(d
\Leb(J)-\epsilon)$ points that are $\frac 1 {M^r}$ separated. Thus
$\{x_{M^{r-1}},x_{M^{r-1}+1},...,x_{M^r}\}$ contains at least $M^r(d
\Leb(J)-\epsilon)-M^{r-1}\geq 4M^{r-1}-M^{r-1}$ points that are
$\frac 1 {M^r}$ separated. To deal with the various $\epsilon$ let
$c =\frac {d\Leb(J)} {6}$, $M=\frac {12} {d \Leb(J)}$ and $e=1$ and
apply Lemma \ref{key}.
\end{proof}
 This provides the following corollary.
\begin{cor} \label{link2} Given a finite union of intervals $J$, if there exists $d>0$ such that $\underset{N \to \infty}{\liminf} \, \Leb(\underset{i=1}{\overset{N}{\cup}} B(x_i, \frac 1 N) \cap J) \geq d \Leb(J)$
then
$\Leb(\Limsup B(x_n, a_n) \cap J)>\constant d\Leb(J)$.  \mbc{What?}
\end{cor}

We now use Lemma \ref{link} to prove Theorem \ref{suff}.
\begin{proof}[Proof of Theorem \ref{suff}] Corollary \ref{link2} shows that every point $y$ satisfies
\[
\tfrac1{2r}\,\Leb(\Limsup B(x_n, a_n) \cap B(y,r))>\constant d>0.
\]
 This implies that $\Leb(\Limsup B(x_n, a_n))=1$ because its complement has no density points.
\end{proof}

\begin{rem} This proof also show that a local version of Theorem \ref{suff} holds. To be exact, let $$f(z):=\underset{ r \to 0^+}{\limsup} \ \underset{N \to \infty}{\liminf} \ \tfrac{\Leb(\underset{i=1}{\overset{N}{\cup}}B(x_i, \frac 1 N) \cap B(z,r))}{ \Leb(B(z,r))}.$$ If $f(z)>0$ for almost every $z$ then $\bf{x}$ is BC.
\end{rem}

Using this remark one can construct sequences that are BC for $\mathbb{R}$.



\begin{proof}[Proof of Theorem \ref{nec suff}]
Assume $\bf{x}$ is not BC. So, there exists a standard sequence $\bf{a}$ so that $\Leb((\Limsup B(x_n, a_n))^c)>0$. We define the following sets.
\begin{center}
$S=\Limsup B(x_n, a_n)$\\
$R_{t,\delta}=\{y \in S^c: \Leb(B(y, \delta') \cap S^c)>2\delta'  t, \ \forall \, \delta'<\delta \}$.
\end{center}
 By the Lebesgue density theorem, $\Leb(S^c \cap (\underset{n=1}{\overset{\infty}{\cup}}R_{t,\frac 1 n}))=\Leb(S^c)$ for any $t<1$.
Choose $\delta$ small enough so that $R_{\frac{999}{1000},\delta}\neq \emptyset$. Let $y_1 \in R_{\frac{999}{1000},\delta}$.
By Lemma \ref{link} there exist infinitely many $N$ such that
\[
\constant \Leb(\underset{i=1}{\overset{N}{\cup}} B(x_i, \tfrac 1 N) \cap B(y_1, \delta)) \leq \tfrac1{1000}\,\Leb( B(y_1, \delta)).
\]
Pick one and denote it $N_1$. We now cover most of $B(y_1,\delta)$.

 There exist points, $y_2^{(1)},y_2^{(2)},...,y_2^{(t_2)}$ and corresponding radii, $r_2^{(1)},r_2^{(2)},...,r_2^{(t_2)}$ such that:
\begin{enumerate}
\item $y_2^{(i)} \in R_{\frac{9999}{10000},r_2^{(i)}}$.
\item $B(y_2^{(i)}, r_2^{(i)}) \subset B(y_1, \delta)$ for all $i$.
\item $\Leb(B(y_1, \delta) \cap (\underset{i=1}{\overset{t_2}{\cup}}B(y_2^{(i)},r_2^{(i)})))>\frac{99}{100}\cdot 2 \delta $.
\item The $B(y_2^{(i)}, r_2^{(i)})$ are all disjoint.
\end{enumerate}
Here is the justification. First notice that
$$\underset{n \to \infty}{\lim} \Leb\left(S^c \cap B(y_1, \delta) \cap R_{t,\epsilon}\right)=\Leb\left(S^c \cap B(y_1,\delta)\right)\geq (1-10^{-3})2\delta$$ and therefore
\[
\Leb \big( B(y_1,\delta) \backslash \underset{n=1}{\overset{\infty}{\cup}}\underset{y \in R_{1-\epsilon,\frac 1 n},B(y,\frac 1 n) \subset B(y_1,\delta)}{\cup}B(y,\tfrac 1 n)\big)\leq (1-10^{-3})2\delta.
\]
 By Theorem \ref{cover} (which gives disjointness of $B(y_2^{(i)},r_2^{(i)})$) it is possible to cover $B(y_1,\delta)$ up to a set of measure $10^{-3}2\delta$ by a countable  number of $B(y_2^{(i)},r_2^{(i)})$ satisfying Conditions 1-4.
 Therefore we can cover all but a set measure $1-10^{-2}$ of $B(y_1,\delta) \cap S^c$ by a finite number of $B(y_2^{(i)},r_2^{(i)})$ satisfying Conditions 1-4.

 By Condition 1 and Corollary \ref{link2} the union of these balls can not have
\[
\constant \Leb(\underset{i=1}{\overset{N}{\cup}} B(x_i, \tfrac 1 N) \cap \underset{j=1}{\overset{t_2}{\cup}} B(y_2 ^{(j)}, r_2^{(j)}))>
10^{-4} \Leb(\underset{j=1}{\overset{t_2}{\cup}} B(y_2 ^{(j)}, r_2^{(j)}))   
\]
 for all but finitely many $N$.
 This implies that for infinitely many $N$, $\frac{99}{100}$ of the measure of the $\underset{j=1}{\overset{t_2}{\cup}}B(y_2^{(j)}, r_2^{(j)})$ have
 \[
 \constant \Leb(\underset{i=1}{\overset{N}{\cup}} B(x_i, \frac 1 N)
 \cap B(y_2^{(j)}, r_2^{(j)}))<10^{-4}\cdot10^{2}\cdot2 \Leb(B(y_2^{(i)}, r_2^{(j)}))
 \]
 at the same time as individual balls. (The constant 2 really depends on how closely we can divide up the measure of my space into balls. One could chose it arbitrarily close to 1.)
 Pick one of these times $N_2$, and the corresponding collection $\mathcal{U}_2$.
  Notice, for any $z \in B(y_2^{(i)},(1-\frac 1 {16}) r_2^{(i)})$ where $i \in \mathcal{U}_2$ we have
\[
\Leb(\underset{j=1}{\overset{N_2}{\cup}}B(x_j, \frac 1 {N_2})\cap B(z, r_2^{(i)}-|z-y_2^{(i)}|))< \constantinv\cdot16\cdot10^{-4}\cdot100\cdot2 \Leb(B(z, r_2^{(i)}-|z-y_2^{(i)}|)).
\]
This follows in the worst case scenario,
\[
\underset{j=1}{\overset{N_2}{\cup}}B(x_j, \frac 1 {N_2}) \cap B(y_2^{(i)}, r_2^{(i)}) \subset B(z, r_2^{(i)}-|y_2^{(i)}-z|).
\]

Also, by condition 3, $\Leb(\underset{i=1}{\overset{N_2}{\cup}} B(x_i, \frac 1 N_2) \cap B(y_1, \delta))>10^{-2}+\constantinv\cdot10^{-4}\cdot2\delta $.

 We now proceed inductively choosing $t_k$ points, $y_k^{(1)},...,y_k^{(t_k)}$ with corresponding radii, $r_k^{(1)},...,r_k^{(t_k)}$ such that:
\begin{enumerate}
\item $y_k^{(i)} \in R_{1-10^{-2k},r_k^{(i)}}$.
\item $B(y_k^{(j)}, r_k^{(j)}) \subset \underset{i=1}{\overset{t_{k-1}}{\cup}} B(y_{k-1}^{(i)}, r_{k-1}^{(i)})$ for all $j$.
\item $\Leb(B(y_{k-1}^{(i)}, r_{k-1}^{(i)}) \cap (\underset{i=1}{\overset{t_k}{\cup}}B(y_k^{(i)},r_k^{(i)})))>(1-10^{-k})\Leb(  B(y_{k-1}^{(i)}, r_{k-1}^{(i)}))$, \\ for each $y_{k-1}^{(i)}$.
\item The $B(y_k^{(i)}, r_k^{(i)})$ are all disjoint.
\end{enumerate}
 This follows analogously to the argument for the existence of $y_2^{(i)},r_2^{(i)}$ satisfying Conditions 1-4 above.

By Condition 1 and Corollary \ref{link2} the union of these balls can not have
$$\constant \Leb(\underset{i=1}{\overset{N}{\cup}} B(x_i, \tfrac 1 N) \cap \underset{j=1}{\overset{t_k}{\cup}}B(y_k^{(j)},r_k^{(j)}))>10^{-2k}\cdot\Leb(\underset{j=1}{\overset{t_k}{\cup}}B(y_k^{(j)},r_k^{(j)}))$$
 for all but finitely many $N$.
 This implies that for infinitely many $N$, $(1-10^{-k})$ of the measure of $B(y_k^{(j)}, r_k^{(j)})$ have
\[
\Leb(\underset{i=1}{\overset{N}{\cup}} B(x_i, \tfrac 1 N) \cap B(y_k^{(j)}, r_k^{(j)}))<\constantinv\cdot10^{-2k}\cdot10^k\cdot2 \Leb(B(y_k^{(j)}, r_k^{(j)}))
\]
\noindent at the same time as individual balls. (As before, the constant 2 depends on how closely we can divide up the measure.) Pick one of these times $N_k$, and the corresponding collection $\mathcal{U}_k$.
 Notice, for any $z \in B(y_k^{(i)},(1-4^{-k}) r_k^{(i)})$ where $y_k^{(i)} \in \mathcal{U}_k$ we have
\vspace{-1mm}
\begin{equation}\label{eq1}
\begin{split}
 \Leb(\underset{j=1}{\overset{N_k}{\cup}} B(x_j, \tfrac 1 {N_k})\cap & B(z, r_k^{(i)}-|z-y_k^{(i)}|))<\\
 &<4^k\cdot\constantinv\cdot10^{-2k}\cdot10^k\cdot 2(r_k^{(i)}-|z-y_k^{(i)}|).
\end{split}
\end{equation}
This follows in the worst case scenario, $$\underset{j=1}{\overset{N_k}{\cup}}B(x_j, \tfrac 1 {N_k}) \cap B(y_k^{(i)},r_i^{(i)}) \subset  B(z, r_k^{(i)}-|z-y_k^{(i)}|).$$

Choose $A=\{N_1,N_2,...\}$. We will show,
$$f_A^{-1}(0) \supset \underset{r=1}{\overset{\infty}{\cup}}
\underset{k=r}{\overset{\infty}{\cap}}\underset{i \in \mathcal{U}_k}{\cup}B(y_k^{(i)}, (1-4^{-k})r_k^{(i)}).$$
 This has positive measure because at each step at most $10^{-k}$ of the measure is eliminated by the choice of $y_k^{(i)}$, $10^{-k}$ is avoided by the choice of $\mathcal{U}_k$ and $4^{-k}$ is avoided by the excluded annuli. If $z \in \underset{r=1}{\overset{\infty}{\cup}}
\underset{k=r}{\overset{\infty}{\cap}}\underset{i=1}{\overset{t_k}{\cup}}B(y_k^{(i)}, (1-4^{-k})r_k^{(i)})$, then for all sufficiently large $k$, there exists $i$ such that $|y_k^{(i)}-z|<(1-4^{-k})r_k^{(i)}$. A sequence of radii tending to zero is given by $r_k^{(i)}-|y_k^{(i)}-z|$. The following lemma and its corollary show that $f_A(z)=0$ in this case.
\begin{lem} Given $\epsilon>0$ there exists $m_{\epsilon}$ such that for all $m>m_{\epsilon}$ 
\[
\Leb(\underset{i=1}{\overset{N_m}{\cup}}B(x_i, \tfrac1{N_m}) \cap B(y_k^{(i)},r_k^{(i)}))<
(\tfrac {10} {9}\cdot10^{-k}+\epsilon)\,2r_k^{(i)}.
\]
\end{lem}
\begin{proof} This follows from the fact that Condition 3 implies
$$ \Leb((\underset{i=1}{\overset{t_m}{\cup}} B(y_m^{(i)},r_m^{(i)}))\cap B(y_k^{(i)},r_k^{(i)}))>(1-\tfrac {10}9\cdot10^{-k})2r_k^{(i)}$$
 for any $m>k$
 and that for large enough $m$, $\constantinv\cdot10^{-m}\cdot2\cdot2r_k^{(i)}+10^{-m}<2\epsilon r_k^{(i)}$ (notice that $10^{-m}$ is greater than or equal to what the choice of $\mathcal{U}_m$ excludes). The lemma follows from Equation \ref{eq1} and assuming that the portion of $B(y_k^{(i)},r_k^{(i)})$ not covered by $B(y_m^{(1)},r_m^{(1)}),...,B(y_m^{(t_m)},r_m^{(t_m)})$ is as large as possible.
\end{proof}
The following corollary is immediate.
\begin{cor} \label{small ball} If $z \in B(y_k^{(i)},(1-4^{-k})r_k^{(i)})$ then for all sufficiently large $m$
\[
\Leb(\underset{i=1}{\overset{N_m}{\cup}}B(x_i, \tfrac 1 {N_m}) \cap B(z,r_k^{(i)}-|z-y_k^{(i)}|))<4^k( \tfrac {10} 9 (10^{-k}+\epsilon))2(r_k^{(i)}-|z-y_k^{(i)}|).
\]
\end{cor}

\vspace{2mm}

Other direction: Assume there exist $A=\{N_1,N_2,...\}$ such that $\Leb{(f_A^{-1}\{0\})>0}$.
 By definition of $f_A$, for each $y \in f_A^{-1}\{0\}$ there exist $r_i(y)$ such that
$$\Leb(\underset{j=1}{\overset{N_k}{\cup}}B(x_j, \tfrac 1 {N_k}) \cap B(y,r_i(y)))<4^{-i} 2r_i(y)$$
 for all $k>k_i(y)$.
 There exists $k_i$ such that for a set $\mathcal{V}_i \subset f^{-1}\{0\}$ with $\Leb(\mathcal{V}_i)>(1-10^{-i})\Leb(f^{-1}\{0\})$ we have that $y \in \mathcal{V}_i$ implies $k_i(y)<k_i$. The sequence $a_i$ is defined by $a_i=\frac 1 {N_{k_j}}$ for $N_{k_{j-1}}<i\leq N_{k_j}$. By our choice of $a_1,a_2,...$ we have that  $\underset{i=1}{\overset{\infty}{\cap}}\mathcal{V}_i$ are not density points for $\underset{r=1}{\overset{\infty}{\cap}}
\underset{i=r}{\overset{\infty}{\cup}}B(x_i,a_i)$. Also $\Leb(\underset{i=1}{\overset{\infty}{\cap}}\mathcal{V}_i) \geq \frac 8 9 \Leb(f^{-1}\{0\})$.
\end{proof}
\begin{rem} The conditions imposed throughout the proof are by no means optimal. Additionally, easier conditions are possible in this case (or the case of
$\mathbb{R}^k$). However, the conditions of this proof generalize to the Ahlfors regular case.
\end{rem}
\begin{rem} Given a standard sequence $\bf a$, one can modify the  argument to find $A$  such that $$\Leb(f_A^{-1}\{0\})>(1-\epsilon) \Leb(\Limsup B(x_n, a_n)).$$ Likewise, given $A$ one can modify the argument to find  a standard sequence $\bf a$ such that $\Leb(\Limsup B(x_n, a_n))> (1- \epsilon)\Leb (f_A^{-1} \{0\}) $.
\end{rem}

\section{Generalizing}

We now generalize to another setting. The results are parallel to the case of $[0,1)$ with Lebesgue measure.

\begin{Bob} A complete metric space $(X,d)$ with measure $\mu$ is called \emph{Ahlfors regular of dimension $\omega$} if there exists a constant $C$ such that for all $y$ and $r \geq 0$, $C^{-1}r^{\omega}<\mu(B(y,r))<Cr^{\omega}$.
\end{Bob}
 For further references on Ahlfors regular spaces see \cite{mattila} or \cite{frac frac}.

\begin{ex} \label{lin rec} A linear recurrent subshift on a finite alphabet (see Example \ref{lin rec IET} or \cite{lin ref}), with metric $\bar{d}(\bf{x},\bf{y})$=$(1+\min\{i: x_i \neq y_i\})^{-1}$ and measure $\mu$ given by the unique measure that the shift map is ergodic under is an Ahlfors regular space (of dimension 1). All minimal substitution dynamical systems are linear recurrent.  \mbc{reference?}
\end{ex}

\begin{ex} Hausdorff $\frac{\log 2}{\log 3}$ measure on the middle thirds cantor set is Ahlfors regular of dimension $\frac{\log 2}{\log 3}$ with respect to the usual metric on \ui.  \mbc{ NO!}
\end{ex}
\begin{Bob}Let $(X,d,\mu)$ be an $\omega$ Ahlfors regular metric space. A sequence ${\bf x}=\{x_n\}$ in $X$ is called Borel-Cantelli (in $X$)
if for any standard sequence $\bf a$
\[
{\mu(\underset{k=1}{\overset{\infty}{\cap}} \underset{n=k}{\overset{\infty}{\cup}} B(x_n, a_n^{\frac 1 {\omega}}))=\mu(X)}.
\]
\end{Bob}
\begin{rem} The Ahlfors regular condition ensures that $\underset{i=1}{\overset{\infty}{\sum}} \mu(B(x_i,a_i^{\frac 1 \omega}))=\infty$ \
if and only if \
$\underset{i=1}{\overset{\infty}{\sum}}a_i=\infty$.
\end{rem}


The following is a sufficient condition for a sequence to be BC in $X$. It is
the version of Theorem \ref{suff} in this (more general) setting.
It is also used in the proof of the necessary and sufficient condition (Theorem \ref{nec suff gen}) in this setting.
\begin{thm}\label{suff gen}
 If\, $\bf{x}$ is a sequence in $X$ and there exists $d>0$ such that for every ball $J$
 the inequality
 \[
 \underset{N \to \infty}{\liminf} \, \mu(\underset{i=1}{\overset{N}{\cup}} B(x_i,( \frac 1 N)^{\frac 1 w}) \cap J) \geq d \mu(J)
 \]
 holds, then $\bf{x}$ is BC in $X$.
\end{thm}
We defer the proof of this theorem to later in the section, after Remark \ref{alph inv}.
\begin{rem} This result also follows from V. Beresnevich, D. Dickson and S. Velani's Corollary 2 in \cite{BDV} with $u(n)=2^n$, $l_n=2^{n-1}$, $\delta=\omega$, $\gamma=0$, $\rho(r)=(\frac 1 {r})^{\frac 1 {\omega}}$. Its proof is included for completeness and the fact that it is used in the proof of Theorem \ref{nec suff gen}.
\end{rem}
\begin{ex} \label{lin rec bc} The systems in Example \ref{lin rec} are ABC
(have the property that all forward orbits are BC).
\end{ex}
\begin{ex} The endpoints of the middle thirds cantor set $K$ (that is, the one sided limit points)
enumerated by increasing denominator form a BC sequence (for~$K$).
\end{ex}
Next we provide a sufficient condition for a sequence not to be Borel-Cantelli.
It is this setting's version of Theorem \ref{nec}.
\begin{thm}If there exists a ball $J$ such that for every $\epsilon>0$ there exists arbitrarily large $N_{\epsilon}$ with $\mu(\underset{i=1}{\overset{N_{\epsilon}}{\cup}} B(x_i, (\frac 1 {N_{\epsilon}})^{\frac 1 {\omega}}) \cap J)< \epsilon \mu (J)$ then $\bf{x}$ is not BC.
\label{nec gen}
\end{thm}
The proof is parallel to Theorem \ref{nec}.
\begin{Bob}Let $A=\{N_1,N_2,...\}$ be an infinite increasing sequence of natural numbers.
 Given $\bf{x}$, define $f_A(z):= \underset{ r \to 0^+}{\liminf} \, \underset{N \in A}{\limsup}
 \frac{\mu(\underset{i=1}{\overset{N}{\cup}}
 B(x_i, (\frac 1 N)^{\frac 1 {\omega}}) \cap B(z,r))} {\mu(B(z,r))}$.
\end{Bob}
\begin{ques} Is $f_A$ measurable? Note that $f_A^{-1}(0)$ is a measurable set.
\end{ques}
\begin{thm}  \label{nec suff gen} $\bf{x}$ is not BC if and only if $f_A^{-1}(0)$ contains a set of positive measure for some~$A$.
\end{thm}
We defer the proof of this theorem to the end of the section.
\begin{rem} \label{ubiq} If one considers Corollary 2 of \cite{BDV} for $\rho(r)= (\frac 1 r) ^{\frac 1 {\omega}}$, $u_n=2^n$, $l_n=2^{n-1}$ and $R_{\alpha}$ one point sets, this theorem provides a necessary and sufficient condition for its conclusion to hold.
\end{rem}
\begin{ex}\label{RV gen} Let $R_1,R_2,...$ be independent random variables all distributed according to a probability measure $\nu$. The sequence $\{R_1(\zeta),R_2(\zeta),...\}$ is BC for $\nu^{\mathbb{N}}$ almost every $\zeta$ iff $\mu \ll \nu$.
\end{ex}
\begin{ex} $T:X \to X$ is continuous, $\mu$ measure preserving and not $\mu$ ergodic then $\mu$ almost every orbit is not Borel-Cantelli. 
\end{ex}
The last theorem of this section is a more general version of \cite[Lemma 9]{cassels} and the proof is largely the same.
\begin{thm} \label{dichot}
Let $\bf s$ be a sequence of real numbers such that
$\underset{n \to \infty} {\lim} s_n=\infty$ and let $\bf x$ be a sequence in $X$.
For almost every $y$ we have \ $\underset{n \to \infty}{\liminf}\ {s_n d(x_n,y)}$ is either zero or infinity. That is,
$\mu(\{y: \underset{n \to \infty}{\liminf}\ {s_n d(x_n,y)} \in (0, \infty)\})=0$.
\end{thm}
To prove this theorem we use a version of the Lebesgue density theorem in this context (see for example \cite[Theorem 1.8]{Hein}), which is also used to prove Theorems \ref{suff gen} and \ref{nec suff gen}.
\begin{thm}  \label{density} For any $\mu $ measurable $A$ there exists $\bar{A}$ such that $\mu(A \Delta \bar{A})=0$ and for every $x \in \bar {A}$ we have $\underset{r \to 0^+}{\lim} \frac{\mu(A \cap B(x,r))}{\mu(B(x,r))}=1$.
\end{thm}
\begin{proof} It suffices to show that $$A_N=\{y: \underset{n \to \infty}{\liminf} \ s_n d(x_n,y) \in (a,2a) \ \text{ and } \  \underset{n>N}{\inf} \ s_nd(x_n,y)>a\}$$ has measure 0. When $s_nd(x_n,y)<2a$ then $B(x_n,\frac a {s_n}) \subset B(y, \frac {3a}{s_n})$ . Notice that when $n>N$ we have $B(x_n,\frac {a}{s_n}) \cap A_N =\emptyset$. But $\mu(B(x_n,\frac {a}{s_n})) \geq  \frac 1 {C^2 3^{\omega}} \mu(B(x_n,\frac {3a}{s_n}))$ implying that $A_N$ has no density points because $\underset{n \to \infty}{\lim} \frac a {s_n}=0$. By Theorem \ref{density}, $A_N$ has measure 0.
\end{proof}

We begin the proof of Theorems \ref{suff gen} and \ref{nec suff gen} with this sections key lemma which is analogous to Lemma \ref{key}.
\begin{lem} Let $M \in \mathbb{N}, c>0, e>0$ be constants,  $\bf{x}$ be a sequence in $X$ and $\bf a$ be a standard sequence.
If there exists a ball $J$ such that for all $r \in \mathbb{N}$ at least $c\mu(J)M^r$ points of the set
$\{x_{M^{r-1}}, x_{M^{r-1}+1},..., x_{M^r}\}$ lie in $J$ and are $(\frac e {M^r})^{\frac 1 {\omega}}$ separated from each other
then there exists $\delta>0$ depending only on  $c$ and $e$ such that ${\mu(\Limsup  B(x_n, (a_n)^{\frac 1 {\omega}}) \cap J)> \delta \mu(J) }$.
\label{Lemma 2 general}
\end{lem}
\begin{proof} The proof is similar to Lemma \ref{key}. As before, we will assume $c \mu(J)>\frac 2 M$. Let $b_i= \min \{(a_{M^j})^{\frac 1 {\omega}}, \frac 1 4 (\frac e {M^j})^{\frac {1}{w}} \}$ for $M^{j-1} \leq i < M^j$. Let $\delta=\frac {ec\mu(J)}{2^{2\omega+1}C}$.
 If  $\mu(\underset{i=k_0}{\overset{M^{j-1}}{\cup}}B(x_i,a_i^{\frac 1 {\omega}}\cap J))<\delta$ then we examine $\underset{i=M^{j-1}}{\overset{M^j}{\cup}}B(x_i,b_i)$.

 By the definition of Ahlfors regular any ball of measure of $m$ contains at most
 $m \frac {M^r}{e} C 2^{2\omega}$ disjoint balls of radius
$\frac 1 4 (\frac e {M^r})^{\frac 1 {\omega}}$.
 Notice that if $\{y_1,...,y_r\}$ is a maximal $\delta$-separated set contained in $J$ then
  $\underset{i=1}{\overset{r}{\cup}}B(y_i,2\delta)$ covers $J$.
 It follows that a $\frac 1 4 (\frac{e}{M^r})^{\frac 1 {\omega}}$
 neighborhood of a ball of measure $m$ contains at most
 $m \frac {M^r}{e} C 2^{2\omega}$ points that are $\frac 3 4 (\frac {e}{M^r})^{\frac 1 {\omega}}$
 separated. Therefore at most
$\delta \frac {M^j}{e} C 2^{2\omega} +M^{j-1} \leq \tfrac32 M^{j-1}$
 of the separated points are within $\frac14(\frac{e } {M^r})^{\frac 1 {\omega}}$
 of the previous measure (if $y_0$ and $y_1$ are $\epsilon$ separated then $B(y_0,\frac {\epsilon}{4})$ is $\frac {\epsilon}{2}$ separated from $B(y_1,\frac{\epsilon}{4})$).
 This leaves at least $(.5c\mu(J)-\frac 1 M)M^j$ separated points left because $c\mu(J)> \frac 2 M$.
 With the observation that $({\frac{c\mu(J)}2-\frac 1 M)\underset{i=M^{j-1}}{\overset{\infty}{\sum}}b_i^{\omega}}$ forms a divergent series the proof is completed.
\end{proof}
This lemma is helpful, because the hypothesis of Theorem \ref{suff gen} imply the hypothesis of the lemma.  \mbc{comment}
\begin{lem} \label{link gen} There exists a function $\beta \colon \mathbb{R}^+ \to \mathbb{R}^+$ with the property that if there exists $d>0$ such that for all $N$ the inequality $\mu(\underset{i=1}{\overset{N}{\cup}} B(x_i, \frac 1 N^{\frac 1 w}) ) \geq d $ holds, then
for any standard sequence $\bf a$ we have
 $\mu(\Limsup  B(x_n, (a_n)^{\frac 1 {\omega}}))> \beta (d)$.
Moreover, $\beta(d)>\big((\frac 1 2) ^{2\omega+1}\frac 1 C \big)(\frac 1 2)^{\omega} \frac{d}{2C}$.
\end{lem}
\begin{proof} Assume $\mu(X)\geq 1$. First observe that $\mu(\underset{i=N}{\overset{N\frac{2C}{d}}{\cup}} B(x_i, (\frac{d}{2CN})^{\frac 1 {\omega}}) )\geq \frac d 2 $
because $\mu(\underset{i=1}{\overset{N}{\cup}} B(x_i, (\frac{d}{2CN})^{\frac{1}{\omega}} ))\leq \frac d 2$.
We claim that at least $(\frac 1 2)^{\omega}N$ of the points in the set $\{x_N,...,x_{\frac{2CN}{d}}\}$ must lie $(\frac{d}{2CN})^{\frac 1 {\omega}}$ apart. To see this, observe that if $y_1,...,y_k$ are a collection of points lying within $(\frac{d}{2CN})^{\frac 1 {\omega}}$ of each other then ${\mu(\underset{i=1}{\overset{k}{\cup}}B(y_i, (\frac d {2CN})^{\frac 1 {\omega}})) \leq 2^{\omega}\frac d {2N}}$. (This union is contained in a ball of radius $2 (\frac d {2CN})^{\frac 1{\omega}}$.) Therefore there must be at least $(\frac 1 2)^{\omega} N $ points in the set $\{x_N,...,x_{\frac{2CN}{d}}\}$ that are $(\frac{d}{2CN})^{\frac 1 {\omega}}$ separated. To sum up, the argument shows $c=(\frac 1 2)^{\omega} \frac{d}{2C}$ and $e=1$.
\end{proof}
The local version follows by realizing that the the previous proof works with ${M=\frac {2C}{d \mu(J)}}$, $c=(\frac 1 2)^{\omega} \frac{d}{2C}$ and $e=1$.
\begin{cor}\label{link2 gen} There exists a function $\alpha: \mathbb{R}^+ \to \mathbb{R}^+$ with the property that if there exists a ball $J$ and $d>0$ such that for all $N$ the inequality $\mu(\underset{i=1}{\overset{N}{\cup}} B(x_i, \frac 1 N^{\frac 1 w})\cap J ) \geq d \mu(J)$ holds, then 
for any standard sequence $\bf a$ we have
 ${\mu(\Limsup  B(x_n, (a_n)^{\frac 1 {\omega}})) \cap J)> \alpha (d) \mu(J)}$.
Moreover, $\alpha(d)>(\frac 1 2 ^{2\omega+1}\frac 1 C )(\frac 1 2)^{\omega} \frac{d}{2C}$.
\end{cor}
\begin{rem} \label{alph inv} For ease of notation let $\alpha^{-1}(s)= \inf \{d>0: \alpha(d)>s\}.$
\end{rem}
\begin{proof}[Proof of Theorem \ref{suff gen}]  By Lemma \ref{Lemma 2 general} the conditions of Theorem \ref{suff gen} imply that for any standard sequence $\bf{a}$ the complement of $\Limsup  B(x_n, (a_n)^{\frac 1 {\omega}})$ has no density points. Theorem \ref{density} implies that $\Limsup  B(x_n, (a_n)^{\frac 1 {\omega}})$ has full measure.
\end{proof}

\begin{rem}\label{thm 3 gen}  
A sequence in  $\{x_1,x_2,...\}$ in  $[0,1)$ is BC if for all $\delta>0$  there are  constants  $\epsilon>0$ and  $M_{\epsilon}$ such that 
$\mu(X \backslash M_{\epsilon})<\delta$ \ and \\[-2mm]
\[
\underset{r \to 0}{\limsup} \ \underset{N \to \infty}{\liminf} \ \mu(\underset{i=1}{\overset{N}{\cup}}B(x_i, (\tfrac 1 N)^{\frac{1}{\omega}})) \cap B(z,r))> \epsilon \mu(B(z,r))
\]
for $z \in M_{\epsilon}$.
\end{rem}
We now continue to the proof of Theorem \ref{nec suff gen}, which requires a covering theorem (see for example \cite[Theorem 1.6]{Hein}).
\begin{thm} \label{cover} Given a measurable set $A$ and family of balls $F$ such that ${\liminf \{r>0: b(a,r) \in F\}=0}$ for all $a \in A$ then one can cover almost all of $A$ by a disjoint countable collection of balls in~ $F$.
\end{thm}
\begin{cor} For any measurable set $A$ with $\mu(A)< \infty$ and $\epsilon>0$ there exists a finite number $N_{A,\epsilon}$ such that one can cover all but $\epsilon$ of $A$ by $N_{A,\epsilon}$ disjoint balls in~ $F$.
\end{cor}
\begin{proof}[Proof of Theorem \ref{nec suff gen}]
Assume $\bf x$ is not BC. Thus, there exists a standard sequence $\bf a$ so that $\mu((\Limsup  B(x_n, (a_n)^{\frac 1 {\omega}}))^c)>0$. We define the following sets:\vspace{-2mm}
\begin{center}
$S=\Limsup  B(x_n, (a_n)^{\frac 1 {\omega}})$\\
$R_{t,\delta}=\{y \in S^c: \mu(B(y, \delta') \cap S^c)>\mu(B(y,\delta))  t \, \forall \, \delta'<\delta \}$.

\end{center}
 By the existence of density points (Theorem \ref{density}), $\mu(S^c \cap \underset{n=1}{\overset{\infty}{\cup}}R_{t,\frac 1 n })=\mu(S^c)$ for any $t<1$.

We also define families of balls (which we will use for covering arguments) by:

$F_{t,s}=\big\{B(y,r)\colon y \in R_{t, r}, \mu(B(y, r) \backslash B(y,(1-s)r ))< 2^{\omega+1} \frac{Cr^{\omega}-\frac {(.5r)^{\omega}} {C}}{\frac 1 {2s}} \big\}$.

\begin{rem} For $t$, all $s$ sufficiently close to 0 and $\mu$ Ahlfors this family satisfies the hypothesis of Theorem \ref{cover} because of Theorem \ref{density} and the fact that there are at least $\frac 1 {2s}-1$ disjoint annuli  $\mu(B(y, r') \backslash B(y,(1-s)r' )) $ between $B(y,\frac 1 2 r)$ and $B(y,r)$. The ball contained in these annuli has radius at least $\frac 1 2 r$. For more sophisticated coverings along this line see Appendix 1 \cite{gdavid}, which describes a generalization of dyadic cubes. This small boundary condition is needed to obtain this setting's version of Corollary \ref{small ball}.
\end{rem}

Armed with the small boundary condition on balls in $F_{t,s}$, we proceed with a similar argument to the case of $\ui$ with $\Leb$ measure.

Choose $\delta$ small enough so that $R_{\frac{999}{1000},\delta}\neq \emptyset$. Let $y_1 \in R_{\frac{999}{1000},\delta}$.
By Corollary \ref{link2 gen} and recalling the definition of $\alpha^{-1}$ in Remark \ref{alph inv}, there exist infinitely many $N$ such that
\[
\mu(\underset{i=1}{\overset{N}{\cup}} B(x_i, (\tfrac 1 N)^{\frac 1 {\omega}}) \cap B(y_1, \delta)) \leq \tfrac1{1000\,\alpha}\,\mu (B(y_1, \delta)).
\]

 Pick one and denote it $N_1$.
 We now proceed directly to the inductive step: Covering most of $B(y_{k-1}^{(1)}, r_{k-1}^{(1)}),...,B(y_{k-1}^{(t_{k-1})}, r_{k-1}^{(t_{k-1})})$ by choosing $t_k$ points $y_k^{(1)},....,y_k^{(t_k)}$ with corresponding radii, $r_k^{(1)},...,r_k^{(t_k)}$ such that:

\begin{enumerate}
\item $B(y_k^{(i)},t_k^{(i)})  \in F_{1-10^{-2k}, (C^{\frac {4} {\omega } } 4^{-k})^{\frac 1 {2} }}$
\item $B(y_k^{(j)}, r_k^{(j)}) \subset \underset{i=1}{\overset{t_{k-1}}{\cup}} B(y_{k-1}^{(i)}, r_{k-1}^{(i)})$ for all $j$.
\item $\mu(B(y_{k-1}^{(i)}, r_{k-1}^{(i)}) \cap (\underset{i=1}{\overset{t_k}{\cup}}B(y_k^{(i)},r_k^{(i)})))>(1-10^{-k})\mu(  B(y_{k-1}^{(i)}, r_{k-1}^{(i)}))$ for each $y_{k-1}^{(i)} $.
\item The $B(y_k^{(i)}, r_k^{(i)})$ are all disjoint.
\end{enumerate}

Here is the justification. First notice that for all $s<1$ we have
$$\underset{n \to \infty}{\lim} \mu(S^c \cap
\underset{i=1}{\overset{t_{k-1}}{\cup}}B(y_{k-1}^{(i)},
r_{k-1}^{(i)}) \cap R_{s,\frac 1 n}) = \mu(S^c \cap
\underset{i=1}{\overset{t_{k-1}}{\cup}}B(y_{k-1}^{(i)},
r_{k-1}^{(i)}))$$ which by induction is greater than
$$(1-10^{-2(k-1)})\mu(
\underset{i=1}{\overset{t_{k-1}}{\cup}}B(y_{k-1}^{(i)},
r_{k-1}^{(i)}))$$ if $k>2$ and $(1-10^{-4})\mu(B(y_1,\delta))$ if
$k=2$. Therefore
$$\mu(\underset{i=1}{\overset{t_{k-1}}{\cup}}B(y_{k-1}^{(i)},
r_{k-1}^{(i)}) \backslash
\underset{n=1}{\overset{\infty}{\cup}}\underset{x \in R_{1-
\epsilon, \frac 1 n} s.t. B(x, \frac 1 n) \subset
\underset{i=1}{\overset{t_{k-1}}{\cup}}B(y_{k-1}^{(i)},
r_{k-1}^{(i)})}{\cup} B(z, \frac 1 n))$$ is at most
$(1-10^{-2(k-1)})\mu(\underset{i=1}{\overset{t_{k-1}}{\cup}}B(y_{k-1}^{(i)},
r_{k-1}^{(i)}))$ if $k>2$ and $(1-10^{-4})\mu(B(y_1,\delta))$ if
$k=2$.

By Theorem \ref{cover} (which gives the disjointness of
$B(y_k^{(i)},r_k^{(i)})$) it is possible to cover $ \underset{i=1}{\overset{t_{k-1}}{\cup}}B(y_{k-1}^{(i)},
r_{k-1}^{(i)})$ up to a set of
measure $(1-10^{-2(k-1)})\mu (\underset{i=1}{\overset{t_{k-1}}{\cup}}B(y_{k-1}^{(i)},
r_{k-1}^{(i)}))$  if $k>2$ and $ (1-10^{-4}) \mu ( B(y_1,\delta))$  if $k=2$ by a countable number of $B(y_k^{(i)},r_k^{(i)})$
satisfying Conditions 1-4. Therefore we can cover all but a set of
measure ${(1-10^{-k})
\mu(\underset{i=1}{\overset{t_{k-1}}{\cup}}B(y_{k-1}^{(i)},
r_{k-1}^{(i)}))}$ by a finite number of $B(y_{k}^{(i)},r_k^{(i)})$
satisfying Conditions 1-4.


By Condition 1 and Corollary \ref{link gen} as the union of these balls can not have
$$\mu(\underset{i=1}{\overset{N}{\cup}} B(x_i, (\frac 1 N)^{\frac 1 {\omega}} )\cap \underset{j=1}{\overset{t_k}{\cup}}B(y_k^{(j)},r_k^{(j)}))>\alpha^{-1}(10^{-2k})\mu( \underset{i=1}{\overset{t_k}{\cup}}B(y_k^{(i)},r_k^{(i)}))$$
 for all but finitely many $N$.
 This implies that for infinitely many $N$, $(1-10^{-k})$ of the measure of $\underset{j=1}{\overset{t_k}{\cup}}B(y_k^{(i)}, r_k^{(i)})$ have
\begin{equation}
\mu(\underset{i=1}{\overset{N}{\cup}} B(x_i, (\tfrac 1 N)^{\frac 1 {\omega}}) \cap B(y_k^{(j)},r_k^{(j)}))<\alpha^{-1}(2\cdot 10^{-k})\mu (B(y_k^{(j)},r_k^{(j)})) \label{fun}
\end{equation}
 at the same time as individual balls. (As before the constant 2 depends on how closely we can divide up the measure.) Pick one of these times $N_k$, and the corresponding collection $\mathcal{U}_k$.
 Notice, for any $z \in B(y_k^{(i)},(1-(C^{\frac{4}{\omega}} 4^{-\frac{k}{\omega}})^{\frac 1 2}) r_k^{(i)})$ where $y_k^{(i)} \in \mathcal{U}_k$ we have:

\[
\mu(\underset{j=1}{\overset{N_k}{\cup}}B(x_j, (\tfrac 1 {N_k})^{\frac 1 {\omega}})\cap B(z, r_k^{(i)}-d(z,y_k^{(i)})))<4^{\frac {k}2}
\cdot \alpha^{-1}(2\cdot10^{-k}) \cdot\mu(B(z, r_k^{(i)}-d(z,y_k^{(i)}))) .
\]

This is obtained by assuming the worst case possible, $$\underset{j=1}{\overset{N_k}{\cup}}B(x_j, (\tfrac 1 {N_k})^{\frac 1 {\omega}}) \cap B(y_k^{(i)},r_k^{(i)}) \subset B(z, r_k^{(i)}-d(z,y_k^{(i)})),$$ $\mu(B(y_k^{(i)},r_k^{(i)}))$ is as large as possible and that $\mu(B(z, r_k^{(i)}-d(z,y_k^{(i)})))$ is as small as possible.

 Our set of times are $A=\{N_1,N_2,...\}$. We will show,
$$f_A^{-1}(0) \supset
 \underset{r=1}{\overset{\infty}{\cup}}
\underset{k=r}{\overset{\infty}{\cap}}
\underset{i\in \mathcal{U}_k}{\overset{}{\cup}}
B(y_k^{(i)}, (1-(C^{\frac 4{\omega}} 4^{-\frac{k}{\omega}})^{\frac 1 2})r_k^{(i)}).$$
 This has positive measure because at each step at most $10^{-k}$ of my measure is kicked out, the choice of $\mathcal{U}_k$ avoids at most $10^{-k}$ of my measure and the annuli only avoids at most $C^2 2^{\omega} (C^{\frac {4} {\omega } } 4^{-\frac{k}{\omega}})^{\frac 1 {2} }$ of my measure (by the definition of $F_{t,s}$). If
$$z \in \underset{r=1}{\overset{\infty}{\cup}}
\underset{k=r}{\overset{\infty}{\cap}}\underset{i=1}{\overset{t_k}{\cup}}B(y_k^{(i)}, 1-(C^{\frac{4}{\omega}} 4^{-\frac{k}{\omega}})^{\frac 1 2}r_k^{(i)}),$$
 then for all sufficiently large $k$, there exists $i$ such that ${|y_k^{(i)}-z|<(1-(C^{\frac{4}{\omega}} 4^{-\frac{k}{\omega}})^{\frac 1 2})r_k^{(i)}}$. A sequence of radii tending to zero is given by $r_k^{(i)}-|y_k^{(i)}-z|$. The following lemma and its corollary show that $f_A(z)=0$ in this case.

\begin{lem} Given $\epsilon>0$ there exists $m_{\epsilon}$ such that for all $m>m_{\epsilon}$
 $$\mu(\underset{i=1}{\overset{N_m}{\cup}}B(x_i, (\frac 1 {N_m})^{\frac 1 {\omega}}) \cap B(y_k^{(i)},r_k^{(i)}))<(\tfrac{10}9\cdot10^{-k}+\epsilon)\mu(B(y_k^{(i)},r_k^{(i)})).$$
\end{lem}
This follows from the fact that by condition 3
$$\mu(\underset{i=1}{\overset{t_m}{\cup}} B(y_m^{(i)},r_m^{(i)})\cap B(y_k^{(i)},r_k^{(i)})>\tfrac{10}{9}\cdot10^{-k}\mu(B(y_k^{(i)},r_k^{(i)})$$
 for any $m>k$
and that for large enough $m$: $$\alpha^{-1}(10^{-m}(r_k^{(i)}))+10^{-m}\mu(B(y_1,\delta))<\epsilon\mu(B(y_k^{(i)},r_k^{(i)})) $$ (notice that $\mathcal{U}_m$ excludes at most $10^{-m}\mu(B(y_1,\delta))$. The lemma follows from equation 2 and assuming the worst possible estimate on the portion not covered by $B(y_m^{(1)},r_m^{(1)}),...,B(y_m^{(t_m)},r_m^{(t_m)})$.

The following Corollary is immediate.
\begin{cor} If $z \in B(y_k^{(i)},(1-(C^{\frac 4{\omega}} 4^{-\frac{k}{\omega}})^{\frac 1 {2}})r_k^{(i)})$ then for all sufficiently large $m$ $$\mu(\underset{i=1}{\overset{N_m}{\cup}}B(x_i, (\tfrac 1 {N_m})^{\frac 1 {\omega}}) \cap B(z,r_k^{(i)}-d(z,y_k^{(i)}))<\mu(B(z,r_k^{(i)}-d(z,y_k^{(i)})))4^{\frac {k}{2}} (\tfrac {10}{9}\cdot10^{-k}+\epsilon).
$$
\end{cor}

\vspace{2mm}

Other direction: Assume there exist $A=\{N_1,N_2,...\}$ such that $f_A^{-1}\{0\} \supset B$ such that $\mu(B)>0$.
 By definition of $f_A$, for each $y \in B$ there exist $r_i(y)$ such that
$$\mu(\underset{j=1}{\overset{N_k}{\cup}}B(x_j, (\tfrac 1 {N_k})^{\frac 1 {\omega}}) \cap B(y,r_i(y)))<\tfrac 1 {4^i} \mu(B(y,r_i(y)))$$
 for all $k>k_i(y)$.
 There exists $k_i$ such that for a set $\mathcal{V}_i \subset B$ with ${\mu(\mathcal{V}_i)>(1-10^{-i})\mu(B)}$ we have that $y \in \mathcal{V}_i$ implies $k_i(y)<k_i$. The sequence $a_i$ is defined by $a_i=\frac 1 {N_{k_j}}$ for $N_{k_{j-1}}<i\leq N_{k_j}$. By our choice of $a_1,a_2,...$ we have that  $\underset{i=1}{\overset{\infty}{\cap}}\mathcal{V}_i$ are not density points for $\Limsup  B(x_n, (a_n)^{\frac 1 {\omega}})$. Also  $\mu(\underset{i=1}{\overset{\infty}{\cap}}\mathcal{V}_i) \geq \frac 8 9 \mu(B)$.
\end{proof}

\section{s-BC}
We now define a modification of Borel-Cantelli sequences which is related to the s-Monotone Shrinking Target Property introduced in \cite{tseng}.
\begin{Bob} A sequence $\bf{x}$ $\subset X$ is called \emph{s-BC} if for any monotonic sequence $\bf{a}$ such that $\underset{i=1}{\overset{\infty}{\sum}} (a_i)^{s}$ diverges we have \ $\Leb(\Limsup\, B(x_n,a_n))$=1.
\end{Bob}
\begin{rem} This property is interesting in the case when $s>1$. In this situation it is weaker than BC (i.e. any BC sequence is s-BC for any s$\geq 1$).
\end{rem}
\begin{lem} \label{transfer} If $s>1$ and $a_1,a_2,...$ is a decreasing sequence such that $\underset{i=1}{\overset{\infty}{\sum}} a_i^s=\infty$, then the sequence $\bf{b}$ given by $b_i=a_{\lfloor i^{s} \rfloor}$ is standard.
\end{lem}
\begin{proof} 
Because the sequence is decreasing, $\underset{i=1}{\overset{\infty}{\sum}}a_i^s=\infty$ iff for any $M \in \mathbb{N}$ we have  $\underset{i=1}{\overset{\infty}{\sum}}M^{i-1}a_{M^i}^s=\underset{i=1}{\overset{\infty}{\sum}}((M^{\frac{1}{s}})^{i-1}a_{M^i})^s= \infty$. Because $s>1$, $\underset{i=1}{\overset{\infty}{\sum}}(M^{\frac{1}{s}})^{i-1}a_{M^i}=\infty$. This is less than or equal to $\underset{i=1}{\overset{\infty}{\sum}} (M^{\frac{1}{s}})^{i-1} b_{(M^{\frac{1}{s}})^{i}}$.
\end{proof}
The proofs of theorems in this section follow from the first section after passing to an appropriate subsequence and are omitted.
\begin{thm} If there exists $d>0$ such that for every interval $J$ the inequality
$\underset{N \to \infty}{\lim} \Leb(\underset{i=1}{\overset{N^s}{\cup}} B(x_i, \frac 1 N) \cap J) \geq d \Leb(J)$
holds, then $\bf{x}$ is $s$-Borel Cantelli.
\end{thm}


The key lemma in this setting is:  \mbc{check}

\begin{lem} \label{key sbc} Let $M\in \mathbb{N}, c>0, e>0$ be constants, let $\bf{x}$ be a sequence in $\ui$ and $\bf{a}$ be a decreasing sequence such that\, $\sum_{n=1}^{\infty}\limits a_n^s=\infty$. If for all $r\in\mathbb{N}$  at least $cM^r$ of the points in the set \
$\{x_{M^{(r-1)s}}, x_{M^{(r-1)s}+1},..., x_{M^{sr}}\}$ are $\frac e {M^r}$ separated from each other,
then there exists $\delta>0$ depending only on $c$ and $e$ such that
\[
\Leb (\Limsup  B(x_n,a_n))> \delta.
\]
\end{lem}

\begin{Bob}Let $A=\{N_1,N_2,...\}$ be an infinite increasing sequence of natural numbers. \mbc{what?}
 Given $\bf{x}$, define $f_A(z):= \underset{ r \to 0^+}{\liminf} \ \underset{ N \in A }{\limsup}\ \frac{\Leb(\underset{i=1}{\overset{N^s}{\cup}}B(x_i, \frac 1 N) \cap B(z,r))} {\Leb(B(z,r))}$.
\end{Bob}
\begin{thm} A sequence is not $s$-BC if and only if there exists a sequence \mbox{$A=\{N_1,N_2,...\}$} such that $\Leb(f_A^{-1}(\{0\}))>0$.
\end{thm}

\begin{rem} The theorems in this section could be rephrased using $\Leb(\underset{i=1}{\overset{N}{\cup}}B(x_i, (\frac 1 N)^s))$. Similarly the theorems in the last section could be rephrased using $\mu (\underset{i=1}{\overset{N^{\omega}}{\cup}}B(x_i, \frac 1 N))$. One can treat $s$-Borel-Cantelli in Ahlfors regular spaces of dimension $\omega$.
\end{rem}
\section{Properties}
For completeness we include some basic properties of BC sequences.
\begin{prop} If $k_1,k_2,...\subset \mathbb{N}$ has positive lower density and $x_{k_1},x_{k_2},...$ is BC then so is $x_1,x_2,...$ .
\end{prop}
\begin{prop} If $x_1,x_2,...$ is Borel-Cantelli and $k_1,k_2,... \subset \mathbb{N}$ has density 1 then $x_{k_1},x_{k_2},...$ is Borel-Cantelli.
\end{prop}
\begin{rem} This proposition states that the property of being Borel-Cantelli survives the deletion of any sequence of density 0. The same need not be true for sequences of positive upper density (even in they have lower density 0).
\end{rem}
\begin{Bob} Given a sequence $\bf x$ we say a measure $\nu$ is a weak-* limit point of $\bf x$ if it is a weak-* limit point of the sequence of measures 
\[
\{\delta_{x_1}, \tfrac 1 2 (\delta_{x_1}+\delta_{x_2}),...,\tfrac 1 N\, \underset{i=1}{\overset{N}{\sum}} \delta_{x_i},...\}
\] 
where $\delta_z$ denotes point mass at $z$.
\end{Bob}
\begin{prop} If $\bf{x}$ is a Borel-Cantelli sequence in $\ui$ then Lebesgue measure is absolutely continuous with respect to its weak-* limit points.
\end{prop}
\begin{proof} Pick a sequence $N_1,N_2,...$ so that $\frac 1 N_j \underset{i=1}{\overset{N_j}{\sum}} \delta_{x_i}$ weak-* converges to a measure $\nu$ that does not have full support. Pick a set $S$ such that $\nu(S)=1$ but $\Leb(S)<1$. By the Lebesgue density theorem and what weak-* convergence means if ${A=\{N_1,N_2,...\}}$ then $\Leb(f_A^{-1}(\{0\}) \cap S^c)=\Leb(S^c)$.
\end{proof}
\begin{rem} The proposition hold for Ahlfors regular spaces as well.
\end{rem}
\begin{Bob} Given a sequence $\bf{x}$ in $(X,\mu)$ an $\omega$ Ahlfors regular space we say $\{y_1,y_2,...\}$ is an $l^p$ perturbation of $\bar{x}$ if $\underset{i =1}{\overset{\infty}{\sum}} d(x_i,y_i)^p$ converges.
\end{Bob}
\begin{prop} If $\bf{x}$ is a Borel-Cantelli sequence in $(X,\mu ,d)$ an $\omega$ Ahlfors regular space then any $l^p$ perturbation for $p \leq \frac 1 {\omega}$ is also a Borel-Cantelli sequence.
\end{prop}
\begin{proof} If $\bf{x}$ is not Borel-Cantelli then there exist $A \subset \mathbb{N}$ and a measurable set $S$  with $\mu(S)>1$ and $f_A(S)=0$. Let $\bf {y}$ be an $l^{\frac 1 {\omega}}$ perturbation of $\bf x$. The same $A$ and $S$ show that $\bf{y}$ is not Borel-Cantelli.
\end{proof}
\section{Acknowledgments}

The authors would like to thank T. DePauw, R. Hardt, D. Kleinbock, H. Kr\"{u}ger, R. Ryham and S. Semmes for helpful conversations. We would like to thank V.~Beresnevich for helpful comments on an earlier version of the paper.

\end{document}